\documentclass[12pt]{amsart}

\usepackage{amscd,amssymb}
\usepackage[all]{xy}
\usepackage[plainpages,backref,urlcolor=blue]{hyperref}
\usepackage[utf8]{inputenc}
\usepackage{mathtools}
\DeclarePairedDelimiter{\floor}{\lfloor}{\rfloor}

\theoremstyle{plain}
\newtheorem{thm}[subsection]{Theorem}
\newtheorem{lem}[subsection]{Lemma}
\newtheorem{prop}[subsection]{Proposition}
\newtheorem{cor}[subsection]{Corollary}

\theoremstyle{definition}
\newtheorem{rk}[subsection]{Remark}

\newtheorem{ex}[subsection]{Example}
\newtheorem{conj}[subsection]{Conjecture}

\numberwithin{equation}{section}
\newcommand{\OO}{\mathcal O}

\newcommand{\al}{\alpha}
\newcommand{\be}{\beta}

\newcommand{\PPP}{\mathcal P}

\newcommand{\Q}{\mathbb Q}

\newcommand{\C}{\mathbb C}
\newcommand{\PP}{\mathbb P}
\newcommand{\N}{\mathbb N}

\DeclareMathOperator{\reg}{reg}

\DeclareMathOperator{\dd}{d}

\begin{document}

\title[Graded Betti numbers of surfaces]{Graded Betti numbers of the Jacobian algebra of surfaces in $\PP^3$}

\author[Alexandru Dimca]{Alexandru Dimca$^1$}
\address{Universit\'e C\^ote d'Azur, CNRS, LJAD, France and Simion Stoilow Institute of Mathematics,
P.O. Box 1-764, RO-014700 Bucharest, Romania}
\email{Alexandru.Dimca@univ-cotedazur.fr}

\author[Gabriel Sticlaru]{Gabriel Sticlaru}
\address{Faculty of Mathematics and Informatics,
Ovidius University
Bd. Mamaia 124, 900527 Constanta,
Romania}
\email{gabriel.sticlaru@gmail.com }

\thanks{$^1$ partial support from the project ``Singularities and Applications'' - CF 132/31.07.2023 funded by the European Union - NextGenerationEU - through Romania's National Recovery and Resilience Plan.}



\begin{abstract}

We compute an explicit
closed formula for the Hilbert polynomial of the Jacobian algebra $M(f)$ of a reduced  surface $X:f=0$ in $\PP^3$ in terms of the graded Betti numbers of the algebra $M(f)$. When $X$ has only isolated singularities, two results by A. du Plessis and C. T. C. Wall  yield new necessary conditions for a set of positive integers to be the graded Betti numbers of the Jacobian algebra of such a surface.
The comparison with the plane curve case is discussed in detail and additional information is given in the case of nodal surfaces. A natural conjecture on the  smallest  4 exponents of $X$ is stated and support for it is provided.
In the final section we construct four natural Jacobian syzygies for surfaces $X$ coming from pencils of surfaces.
\end{abstract}

\maketitle

%%%%%%%%%%%%%%%%%%%%%%%%%%%%%%%%%%%%%%%%%%%%%%%%%%%%
\section{Introduction}
%%%%%%%%%%%%%%%%%%%%%%%%%%%%%%%%%%%%%%%%%%%%%%%%%%%%

Let $S=\C[x,y,z,t]$ be the polynomial ring in four variables $x,y,z,t$ with complex coefficients, and let $X:f=0$ be a reduced surface of degree $d\geq 3$ in the complex projective space $\PP^3$. 
We denote by $J_f$ the Jacobian ideal of $f$, i.e. the homogeneous ideal in $S$ spanned by the partial derivatives $f_x,f_y,f_z,f_t$ of $f$, and  by $M(f)=S/J_f$ the corresponding graded quotient ring, called the Jacobian (or Milnor) algebra of $f$.
Consider the general form of the minimal resolution of the Milnor algebra $M(f)$
of a reduced surface $X:f=0$ 
\begin{equation}
\label{res2A}
0 \longrightarrow \bigoplus_{k=1}^{r} S(1-d-b_k)
 \longrightarrow \bigoplus_{j=1}^{q} S(1-d-c_j)
   \longrightarrow \bigoplus_{i=1}^{p} S(1-d-d_i)
   \longrightarrow S^4(1-d)
   \longrightarrow S.
\end{equation}
where $p \geq 3$, $q \geq 0$ and $r \geq 0$.
We call the ordered sequence of degrees 
$${\bf d}=(d_1, \ldots, d_p), \  
{\bf c}=(c_1, \ldots, c_q) 
 \text{ and } {\bf b}=(b_1, \ldots, b_r)$$ 
  the {\it graded Betti numbers of the Jacobian algebra} $M(f)$, since they determine and are determined by the usual
graded Betti numbers of the Jacobian algebra $M(f)$ as defined for instance in \cite{Eis}. An interesting question is to find many, ideally all, independent relations satisfied by these Betti numbers, see Remark \ref{rk13} below.

It is known that there is a unique polynomial $P(M(f))(u) \in \Q[u]$, called the {\it Hilbert polynomial} of $M(f)$, and an integer $k_0\in \N$ such that
\begin{equation}
\label{Hpoly}
 \dim M(f)_k= P(M(f))(k)
\end{equation}
for all $k \geq k_0$. We denote by $\Sigma$ the singular subscheme of $X$, which is defined by the Jacobian ideal $J(f)$. The general theory of Hilbert polynomials says that the degree of  $P(M(f))$ is given by the dimension of the support of  $\OO_{\Sigma}$, the coherent sheal associated to the graded $S$-module $ M(f)$. Hence the assumption $\dim \Sigma=0$ implies that the polynomial
$P(M(f))$ is a constant, namely the total Tjurina number of $X$, given by
\begin{equation}
\label{ab0}
P(M(f))=\tau(X)=\sum_{s \in \Sigma} \tau(X,s),
\end{equation}
where $\tau(X,s)$ denotes the Tjurina number of the isolated singularity $(X,s)$, and
$\dim \Sigma=1$ implies that 
\begin{equation}
\label{ab1}
P(M(f))(u)=au+b,
\end{equation}
where $a=\deg(\Sigma)$, the degree of the subscheme $\Sigma$. 

The first main result of this note is the following computation of the Hilbert polynomial $P(M(f))$ in terms of the graded Betti numbers of the Jacobian algebra $M(f)$
introduced in \eqref{res2A}. In the claim (3) below, as well as in the sequel of this paper, we use  the convention that $u_v$ means $u$ is repeated $v$ times in a list of Betti numbers.

 \begin{thm}
\label{thm1}
For the minimal resolution \eqref{res2A} of the Jacobian algebra $M(f)$ of a reduced surface $X:f=0$ of degree $d$ in $\PP^3$, one has the following.
\begin{enumerate}

\item For any such surface, one has 
$$p+r=q+3 \text{ and }  \sum_{i=1}^pd_i- \sum_{j=1}^qc_j+ \sum_{k=1}^rb_k=d-1.$$

\item The surface $X$ has at most isolated singularities if and only if
$$(d-1)^2+\sum_{i=1}^pd_i^2-\sum_{j=1}^qc_j^2+\sum_{k=1}^rb_k^2=0.$$
If this is the case, then the total Tjurina number of $X$ is given by the formula
$$6\tau(X)=(d-1)^3-\sum_{i=1}^pd_i^3+\sum_{j=1}^qc_j^3-\sum_{k=1}^rb_k^3.$$

\item The surface $X$ is smooth if and only if $p=6$, $q=4$, $r=1$ and
$${\bf d}=(d-1)_6, \ {\bf c}=(2(d-1))_4 \text{ and }  {\bf b}=3(d-1).$$
\item If the surface $X:f=0$ has a 1-dimensional singularity subscheme $\Sigma$, then the Hilbert polynomial $P(M(f))$ of the  Jacobian algebra $M(f)$ is given by
$$P(M(f))(u)=\frac{A}{2}u -B,$$
where
$$A=  (d-1)^2+\sum_{i=1}^pd_i^2-\sum_{j=1}^qc_j^2+\sum_{k=1}^rb_k^2     $$
and
$$B= \frac{d(d-3)^2-4}{3}+\frac{(d-3)}{2}\left(\sum_{i=1}^pd_i^2-\sum_{j=1}^qc_j^2+\sum_{k=1}^rb_k^2 \right)+\frac{1}{6}\left(\sum_{i=1}^pd_i^3-\sum_{j=1}^qc_j^3+\sum_{k=1}^rb_k^3\right).  $$

\end{enumerate}

\end{thm}
When the surface $X$ has only isolated singularities, that is when $\dim \Sigma =0$, then using two results by A. du Plessis and C.T.C. Wall quoted below in Theorem \ref{thmC} and Theorem \ref{thmD}, we obtain the following new restrictions on the
graded Betti numbers of the Jacobian algebra $M(f)$.

\begin{cor}
\label{corC} 
If the surface $X$ has  isolated singularities,
then
$$5(d-1)^3-6d_1(d-1)^2 \leq  -\sum_{i=1}^pd_i^3+\sum_{j=1}^qc_j^3-\sum_{k=1}^rb_k^3 \leq 5(d-1)^3-6d_1(d-d_1-1)(d-1).$$
Moreover, if $d_1+d_2 \ne d-1$, in particular if $d_1>(d-1)/2$, then the last inequality is strict.
\end{cor}
Theorem \ref{thm1} is proved in Section 2 and Corollary \ref{corC} is proved in Section 3. In
Remark \ref{rk13} we recall first the known restrictions on the Betti numbers of a surface. Then we shows the interest of Corollary \ref{corC} in deciding if a set of positive integers may be the graded Betti numbers of a surface $X$ with isolated singularities.

In Section 4 a comparison with the plane curve case is discussed in detail. Surprizingly, many facts holding for the graded Betti numbers of curves fail in the case of surfaces, see for instance Propositions \ref{propT0} and \ref{propT}, Corollary \ref{corT} in relation with the type $t(X)$ of the surface $X$ which is defined by
$$t(X)=d_1+d_2+d_3-d+1$$
in analogy to the curve case studied in \cite{ADP,Type3}.
In particular, Theorem \ref{thmD} implies that for a surface of degree $d$ with isolated singularities one has
$$t(X) \geq \frac{d-1}{2},$$
see for details Proposition \ref{propT0} below.

In Proposition \ref{propN} we give information on the exponents of a nodal surface $X$ using Hodge theory. This result in particular implies that for such a surface $X$ one has $$t(X) =2(d-1).$$
As explained below in Section 4, by analogy to the curve case, we put forth the following.
 \begin{conj}
\label{C1}
 For any reduced surface $X:f=0$ of degree $d$ which is not free one has
 $$d_1 \leq d_2 \leq d_3 \leq d_4 \leq d-1.$$
  \end{conj}
In Remark \ref{rkC1} we list 3 classes of surfaces for which this conjecture holds. Moreover, in Proposition \ref{propC1} we prove a weaker result, namely that one has
 $$d_1 \leq d_2 \leq d_3  \leq d-1$$
for any surface.

In Section 5 we collect a number of additional examples showing the differences between the curve and the surface cases. All the minimal resolutions  corresponding to \eqref{res2A} are computated in this note using the
Computer Algebra softwares CoCoA \cite{CoCoA} and SINGULAR \cite{Singular}.

In the final section we construct four natural Jacobian syzygies for surfaces $X$ coming from pencils of surfaces, similar to the case of plane curves described in \cite[Section 4]{Mich}. This construction allows us to construct surfaces $X$ having the type 
$$t(X)=d_1+d_2+d_3+1-d<-P,$$ where $P$ is an arbitrary large positive number, see Example \ref{ex6.1}. We would like to thank Piotr Pokora for bringing this example to our attention.

\section{Proof of Theorem \ref{thm1}}

We start with the following two Lemmas, whose proofs are elementary and straightforward, so we leave them to the reader.
\begin{lem}
\label{lem1}
For any integer $a$ and $k > |a|$, one has
$$\dim S_{k+a}=\binom{k+a+3}{3}=\frac{k^3+3(a+2)k^2+(3a^2+12a+11)k+a^3+6a^2+11a+6}{6}.$$
\end{lem}
Using the resolution \eqref{res2A}, we get the following
\begin{equation}
\label{res2C1}
 \dim M(f)_{s+d-1}= \dim S_{s+d-1}-4 \dim S_s+\sum_{i=1}^p\dim S_{s-d_i}-\sum_{j=1}^q\dim S_{s-c_j}+\sum_{k=1}^r\dim S_{s-b_k},
\end{equation}
for any $s$ large enough.
Using now Lemma \ref{lem1}, we get the following.

\begin{lem}
\label{lem2}
With the above notation, for any large integer $s$, one has the equality
 $$6 \dim M(f)_{s+d-1}=$$
 $$=(p-q+r-3)s^3+3(d-1-d_1-d_2-d_3-\sum_{i=4}^p(d_i-2)+\sum_{j=1}^q(c_j-2)-\sum_{k=1}^r(b_k-2))s^2+
 $$
 $$+  (3d^2+6d+2-44+\sum_{i=1}^p(3d_i^2-12d_i+11)-  \sum_{j=1}^q(3c_j^2-12c_j+11) +\sum_{k=1}^r(3b_k^2-12b_k+11) )s+ $$
 $$+d^3+3d^2+2d-24-\sum_{i=1}^p(d_i^3-6d_i^2+11d_i-6)+\sum_{j=1}^q(c_j^3-6c_j^2+11c_j-6)-\sum_{k=1}^r(b_k^3-6b_k^2+11b_k-6).$$
\end{lem}
Using these computations, we can prove Theorem \ref{thm1} as follows.
By definition of the Hilbert polynomial $P(M(f))$ we have
$$6P(M(f))(s+d-1)=6 \dim M(f)_{s+d-1}.$$
Since the surface $X$ is reduced, we have
$$\deg P(M(f))=\dim \Sigma \leq 1,$$
and hence the coefficients of $s^3$ and of $s^2$ in Lemma \ref{lem2} must vanish. This proves the claim (1).

To prove the claim (2), we note that $X$ has at most isolated singularities if and only if $\dim \Sigma <1$.  This last condition is equivalent to the vanishing of the coefficients of $s$ in Lemma \ref{lem2}, in addition to the vanishings from the claim (1).
When all these vanishing holds, then we conclude the proof of claim (2) by using \eqref{ab0} and the expression of the constant term in 
 Lemma \ref{lem2}, simplified by using the equalities in (1) and the first equality in (2).
 
 To prove the claim (3), assume first that $X:f=0$ is smooth. Then the partial derivatives $f_x,f_y,f_z$ and $f_t$ form a regular sequence in $S$ and the resolution of $M(f)$ is well known in this case, and has the form
 $$0 \to S(4-4d) \to S(3-3d)^4 \to S(2-2d)^6 \to S(1-d)^4 \to S.$$
It follows that  ${\bf d}$, $ {\bf c}$ and $ {\bf b}$ are given by the equalities in claim (3) when $X$ is smooth. Conversely, if ${\bf d}$, $ {\bf c}$ and $ {\bf b}$ are given by the equalities in claim (3), then using the claim (2) we see that $X$ has at most isolated singularities and that $\tau(X)=0$. Therefore the surface $X$ is smooth.

The proof of claim (4) follows directly from \eqref{ab1} and  Lemma \ref{lem2}, where $s$ has to be replaced by $u-(d-1)$.

%%%%%%%%%%%%%%%%%%%%%%%%%%%%%%%%%%%%%%%%%%%%%%%%%%%%

%%%%%%%%%%%%%%%%%%%%%%%%%%%%%%%%%%%%%%%%%%%%%%%%%%%%
\section{Proof of Corollary \ref{corC} and some remarks}
%%%%%%%%%%%%%%%%%%%%%%%%%%%%%%%%%%%%%%%%%%%%%%%%%%%%

One has the following result, see \cite[Theorem 5.3]{duPCTC01}.
\begin{thm}
\label{thmC}
If the surface $X$ has at most isolated singularities,
then
$$(d-1)^3-d_1(d-1)^2 \leq \tau(X) \leq (d-1)^3-d_1(d-d_1-1)(d-1).$$
Moreover the upper bound for $\tau(X)$ is attained if and only if 
$$d_1+d_2=d-1.$$
\end{thm}
Related to the last claim, one should also recall  the following key result, see \cite[Lemma 5.2]{duPCTC01}.
\begin{thm}
\label{thmD}
If the surface $X$ of degree $d$ has at most isolated singularities,
then
$$d_1+d_2 \geq d-1,$$
where $d_1$ and $d_2$ are the smallest two exponents of $X$.
\end{thm}
As an application of Theorem \ref{thmD}, see Proposition \ref{propT0} below.

To prove Corollary \ref{corC} it is enough to replace in Theorem \ref{thmC} $\tau(X)$ by the value of it given by Theorem \ref{thm1} (2), multiply by 6 and simplify one term $(d-1)^3$.

In the case of plane curves, the corresponding result was obtained in \cite{duPCTC}, 
and played a key role in the understanding of free curves. Indeed, the reduced curve $C$ is free if and only if 
$$\tau(C)=(d-1)^2-d_1(d-d_1-1),$$
i.e. the upper bound is attained, see \cite{Dmax, E} for related results.
A free surface $X$ has necessarily non-isolated singularities,
and so freeness must be related to other invariants, see for instance
\cite{DmaxS}.

\begin{rk}
\label{rkC} 
The lower bound in Theorem \ref{thmC} is attained for any pair $(d,d_1)$.
Indeed, it is enough to find a degree $d$, reduced curve $C: f'(x,y,z)=0$ such that $d_1$ is the minimal exponent of $C$ and
$$\tau(C)=(d-d_1-1)(d-1),$$
and then take $X:f=0$, with 
$$f(x,y,z,t)=f'(x,y,z)+t^d.$$
The existence of curves $C$ as above is shown in \cite[Example 4.5]{3syz} and a complete characterization of them is given in \cite[Theorem 3.5 (1)]{3syz}.
\end{rk}

\begin{rk}
\label{rkC2} 
The upper bound in Theorem \ref{thmC} is attained for any pair $(d,d_1)$ with $2d_1<d$, since for such pairs $(d,d_1)$ the existence of free plane curves $C:f'=0$ of degree $d$ and with exponents $(d_1,d_2)$ is shown in
\cite{DStExpo} and then one constructs the surface $X$ as in Remark \ref{rkC} above.
It is an interesting {\it open question} to improve the upper bound in Theorem \ref{thmC} when $2d_1 \geq d$. The best upper bound for such pairs is (at least conjecturally) known in the case of plane curves, see  \cite{duPCTC, maxTjurina}, and is given by the stronger inequality
$$\tau(C) \leq (d-1)^2-d_1(d-d_1-1)-{2d_1+2-d \choose 2}.$$
If  we start with a degree $d$ and a reduced curve $C: f'(x,y,z)=0$ such that $d_1 \geq d/2$ and
 take $X:f=0$, with $$f(x,y,z,t)=f'(x,y,z)+t^d,$$
 then $f$ and $f'$ have the same minimal exponent $d_1$ and
 \begin{equation} 
\label{betterb} 
\tau(X) \leq (d-1)^3-d_1(d-d_1-1)(d-1)-{2d_1+2-d \choose 2}(d-1).
\end{equation}  
However, this stronger inequality fails for surfaces not constructed as suspensions of plane curves, as the following examples show.
\end{rk}
\begin{ex}
\label{ex10}
Consider the Cayley surface
$$X:f=xyz+xyt+xzt+yzt =0$$
in $\PP^3$ having four $A_1$-singularities. Then $d=3$, $d_1=2 >d/2$ and 
$\tau(X)=4$. Indeed, the minimal resolution of the Jacobian algebra is given by
$$0 \to S[-6]^2 \to S[-5]^8 \to S[-4]^9 \to S[-2]^4 \to S$$
and hence 
$${\bf d}=(2_9), \ {\bf c} =(3_8) \text{ and } {\bf b}=(4_2).$$
The inequality in Theorem \ref{thmC} is in this case
$0 \leq \tau(X) \leq 8,$
while the bound given by \eqref{betterb} is $2$, which is clearly not good.
\end{ex}

\begin{ex}
\label{ex20}
Consider the Kummer surface
$$X:f=x^4+y^4+z^4+t^4-y^2 z^2-z^2 x^2-x^2 y^2-x^2 t^2-y^2 t^2-z^2 t^2 =0$$
in $\PP^3$ having sixteen $A_1$-singularities. Then $d=4$, $d_1=3 >d/2$ and 
$\tau(X)=16$. Indeed, the minimal resolution of the Jacobian algebra is given by
$$0 \to S(-8)^3\to S(-7)^{12} \to S(-6)^{12} \to S(-3)^4\to S\to 0$$
and hence 
$${\bf d}=(3_{12}), \ {\bf c} =(4_{12}) \text{ and } {\bf b}=(5_3).$$
The inequality in Theorem \ref{thmC} is in this case
$0 \leq \tau(X) \leq 27,$
while the bound given by \eqref{betterb} is $15$, which is clearly not good.
\end{ex}
\begin{ex}
\label{ex30}
Consider finally the octic  surface with 144 nodes 
$$X:f=16(x^8+y^8+z^8+t^8)+224(x^4y^4+x^4z^4+x^4t^4+y^4z^4+y^4t^4+z^4t^4)+$$
$$+
2688x^2y^2z^2t^2-9(x^2+y^2+z^2+t^2)^4=0.$$
This is a special case of Chebyshev hypersurfaces, which are classical examples of
nodal hypersurfaces with many singularities. They were introduced by Chmutov to construct complex projective
hypersurfaces with a large number of nodes, see \cite{AGV} Vol. 2, p. 419, \cite{Chm} as well as \cite[Corollary 3.2, (iii)]{DStChm}.
The minimal resolution of the Jacobian algebra is given by
$$0\to S(-20)^4 \oplus S(-22) \to S(-17)^4 \oplus S(-18)^{13} \to S(-14)^6\oplus S(-16)^9\to S(-7)^4\to S$$
and hence 
$${\bf d}=(7_6,9_9), \ {\bf c} =(10_4,11_{13}) \text{ and } {\bf b}=(13_4,15).$$
The inequality in Theorem \ref{thmC} is in this case
$0 \leq \tau(X)=144 \leq 343,$
while the bound given by \eqref{betterb} is $147$, which is also good.
This example is a motivation for Proposition \ref{propN} below.
\end{ex}

\begin{rk}
\label{rk13} 
The graded Betti numbers $\bf d$, $\bf c$ and $\bf b$ satisfy some obvious relations.
First, one clearly has
\begin{equation}
\label{eqR1}
 d_1=\min _i d_i \geq 0,
\end{equation}
with equality if and only if $X$ is a cone over a plane curve.
If $q>0$, that is if $X$ is not a free surface, a secondary syzygy must involve at least three primary syzygies.
Indeed, let $r_1$ and $r_2$ be part of a minimal generating set of the first order syzygies of $f$. Here minimal means that both the numbers of the generators and their degrees are minimal. Assume that a minimal degree second order
syzygy has the form
$$a_1r_1+a_2r_2=0,$$
with $a_1,a_2$ both nonzero.
Then $\deg a_1 >0$, since otherwise the generator $r_1$ is not needed.
Moreover, $a_1$ and $a_2$ are relatively prime, since otherwise we may simplify by any common factor and get a second order syzygy of strictly lower degree. Since $a_1$ and $a_2$ are relatively prime, it follows that
$a_1$ divides all the components of $r_2\in S^4$. Hence $r_2$ is not a primitive syzygy, and hence it can be replaced by a syzygy $r_2'$  of strictly lower degree. This contradiction proves our claim that any second order syzygy involves at least three first order syzygies $r_1,r_2$ and $r_3$.
Using  the minimality of the resolution \eqref{res2A} and the fact that the Betti numbers are ordered, namely
$$d_1 \leq d_2 \leq \ldots \leq d_p, $$
we deduce that
\begin{equation}
\label{eqR2}
 c_1=\min _j c_j >\max\{d_1,d_2,d_3\}.
\end{equation}
If $r>0$, we have in a similar way that $q \geq 3$ and
\begin{equation}
\label{eqR3}
 b_1=\min _k b_k >\max\{c_1,c_2,c_3\}.
\end{equation}
When the singular surface $X$ of degree $d\geq 3$ has at most isolated singularities, we also have that the  Castelnuovo-Mumford regularity of $M(f)$ satisfies
$\reg M(f) \leq 3(d-3)$, see  \cite[Proposition 2.3]{Hess}. In terms of graded Betti numbers this yields the following inequalities
\begin{equation}
\label{eqR4}
 d_p=\max _i d_i \leq 3(d-2), \ c_q=\max c_j \leq 3(d-2)+1 \text{ and } b_r=\max b_k \leq 3(d-2)+2.
\end{equation}
 We note that \eqref{eqR4} is not an arithmetical consequence of the equalities (1) and (2) in Theorem \ref{thm1} and of the inequalities \eqref{eqR1},  \eqref{eqR2} and  \eqref{eqR3}. Indeed, if we choose
$${\bf d}=(2_3,18) \text{ and } {\bf c}=(19),$$
then the equalities (1) and (2) in Theorem \ref{thm1} are satisfied for $d=6$ and clearly the inequalities \eqref{eqR1},  \eqref{eqR2} and  \eqref{eqR3} also hold. On the other hand, the first inequality \eqref{eqR4} fails for this choice, since $d_p=18>12$.
In this example one has $d_1+d_2=2+2 <d-1=5$, hence Theorem \ref{thmD} can also be used to show that a surface $X$ with such Betti numbers cannot have isolated singularities.

 For surfaces with isolated singularities, one may ask if the inequality in Corollary \ref{corC} is a consequence of the equalities (1) and (2) in Theorem \ref{thm1} and of the inequalities \eqref{eqR1},  \eqref{eqR2},  \eqref{eqR3} and  \eqref{eqR4}.
If we choose ${\bf d}=(2_2,3,4)$ and ${\bf c}=(7)$,
then the equalities (1) and (2) in Theorem \ref{thm1} are satisfied for $d=5$. With this choice, the inequalities \eqref{eqR1},  \eqref{eqR2},  \eqref{eqR3} and  \eqref{eqR4}, as well as the inequality in Theorem \ref{thmD}, are also satisfied. On the other hand, for this choice of $\bf b$ and $\bf c$, a direct computation shows that Corollary \ref{corC} yields  the inequalities
$$128 \leq 236 \leq 224.$$
It follows that there is no surface $X$ having only isolated singularities and having as graded Betti numbers of its Jacobian algebra the sequences
$${\bf d}=(2_2,3,4) \text{ and }{\bf c}=(7).$$
\end{rk}

\section{Comparison to the curve case}
Let $R=\C[x,y,z]$ and 
consider  the general form of the minimal resolution of the Milnor algebra $M(f')$
of a reduced plane curve $C:f'=0$, which is assumed not to be free
\begin{equation}
\label{res2A2}
0 \longrightarrow \bigoplus_{j=1}^{q'} R(1-d-c'_j)
   \longrightarrow \bigoplus_{i=1}^{p'} R(1-d-d'_i)
   \longrightarrow R^3(1-d)
   \longrightarrow R,
\end{equation}
with $c'_1\leq ...\leq c'_{q'}$ and $d'_1\leq ...\leq d'_{p'}$, where $d=\deg f'$, see for instance \cite{HS}. One has
$$p'=q'+2,$$
which corresponds to the first equality in Theorem \ref{thm1} (1) above.
It follows from \cite[Lemma 1.1]{HS} that
\begin{equation}
\label{res2B}
 \epsilon_j=  d'_{j+2} -c'_j \ge 1
\qquad j = 1,\dots,q'.
\end{equation}

Using \cite[Formula (13)]{HS}, one obtains the relation
\begin{equation}
\label{res2C}
d'_1 + d'_2 = d - 1 + \sum_{j=1}^{q'} \epsilon_j
\end{equation}
or, equivalently,
\begin{equation}
\label{res2C2}
\sum_{i=1}^{p'}d_i-\sum_{j=1}^{q'}c'_j=d-1,
\end{equation}
which corresponds to the second equality in Theorem \ref{thm1} (1) above. 

A new invariant was recently introduced for a reduced plane curve $C$, namely the type of $C$ defined by
 \begin{equation} 
\label{t1} 
t(C)=d'_1+d'_2-d+1, 
 \end{equation} 
 see \cite{ADP}. When the curve $C$ is not free, then one clearly has
 \begin{equation} 
\label{t11} 
t(C)= \sum_{j=1}^{q'} \epsilon_j>0.
 \end{equation}  
 This invariant has nice properties, for instance a 
 reduced plane curve $C$ is free (resp. plus-one generated) if and only if
 $t(C)=0$ (resp. $t(C)=1$). More, one has $t(C) \geq 0$ for any reduced curve, and the curves with $t(C)=2$ and $t(C)=3$ have been studied in detail in \cite{ADP, Type3}.
 One may try to extend this invariant to surfaces $X$ in $\PP^3$ by setting
 \begin{equation} 
\label{t2}
t(X)=d_1+d_2+d_3+1-d.
 \end{equation} 
 It is clear that if the surface $X$ is free, then $t(X)=0$, but the converse result fails in general, see \cite[Theorem 4.1 (4) and Example 4.8]{HD}.
 Here is a first result on this new invariant.
   \begin{prop} 
\label{propT0}
Let $X:f=0$ be a surface in $\PP^3$ with at most isolated singularities of degree $d$. Then
$$t(X)=d_1+d_2+d_3+1-d \geq d_3 \geq \frac{d-1}{2}.$$
In particular, if $d>1$, such a surface is not free.
  \end{prop} 
  \proof 
  Using Theorem \ref{thmD} we have
  $$t(X) \geq (d-1)+d_3+1-d=d_3 \geq d_2 \geq \frac{d_1+d_2}{2} \geq \frac{d-1}{2}.$$
   \endproof
   \begin{cor} 
\label{corT0} 
If the surface $X$ has  at most isolated singularities, then $X$ cannot be a strictly plus-one generated surface.
 \end{cor}  
 \proof
 Using \cite[Theorem 4.5]{HD}, we see that for a strictly plus-one generated surface one has
 $$d_1+d_2+d_3=d$$
 where the exponents $d_1,d_2$ and $d_3$ are not necessarily the minimal ones. However, Theorem \ref{thmD} implies that, if the above equality holds, then $d_1=d_2=d_3=1$ and hence $d=3$.
 On the other hand, we know that such a cubic surface does not exist, see \cite[Corollary 2.7]{duPCTC08}.
 \endproof
  \begin{ex} 
\label{exT0}  
 As explained in \cite[Corollary 5.2.1 and Theorem 5.7]{duPCTC01}, see also the discussion at the end of section 5 \cite{duPCTC01}in a cubic surface $X$ can have at most two exponents equal to 1. This happens if and only if $\tau(X)=6$, for instance for
 the surface
 $$X_1: f_1=xyz+t^3=0$$
 having 3 singularities of type $A_2$, and for the surface
 $$X_2:f_2=tx^2+xz^2+y^3=0$$
 having a singularity of type $E_6$.
 In both cases, the exponents are $$d_1=d_2=1 <d_3=d_4=d_5=2.$$
  \end{ex}
 The next result shows that the case of surfaces with a 1-dimensional singular locus is much more complicated.
  \begin{prop} 
\label{propT}
Let $X:f=0$ be a reduced surface in $\PP^3$ with the minimal resolution \eqref{res2A}. Then
$$t(X)=d_1+d_2+d_3+1-d \geq 0$$
if there are first order Jacobian syzygies $\rho_1, \rho_2$ and $\rho_3$
which are linearly independent over the field of fractions $K$ of the polynomial ring $S$ and
such that $\deg \rho_j=d_j$ for $j=1,2,3$. 
 \end{prop} 
Note that if the surface $X$ is tame with respect to the pair
$(\rho_1, \rho_2)$ as in \cite[Definition 1.2]{HD}, then  the first order Jacobian syzygies $\rho_1, \rho_2$ and $\rho_3$
 are linearly independent over the field $K$ for any new additional syzygy $\rho_3$. However, unlike the curve case, Example \ref{ex6.1} shows that in general one may have $t(X)<0$.
 
 \proof
 Let $M(E,\rho_1,\rho_2,\rho_3)$ be the $4 \times 4$ matrix with the first row $(x,y,z,t)$ and the $j$-th row, for $j=2,3,4$ being given by the components of the  syzygy $\rho_{j-1}$. Then 
 $$g=\det M(E,\rho_1,\rho_2,\rho_3) \ne 0$$
 if and only if  the syzygies $\rho_1, \rho_2$ and $\rho_3$
 are linearly independent over the field $K$. On the other hand, at any
 smooth point $s \in X$, the rows of $M(E,\rho_1,\rho_2,\rho_3)$ evaluated at $s$ are tangent vectors to $X$ at $s$. Since their number is larger than the dimension of this vector space, it follows that $g(s)=0$
 at any smooth point of $X$. Therefore $g$ vanishes on $X$, and hence $g$ is divisible by $f$, the surface $X$ being reduced.
 It follows that
 $$1+d_1+d_2+d_3=\deg g \geq \deg f=d,$$
 which proves our claim.
 \endproof
 
    \begin{cor} 
\label{corT} 
For any reduced surface $X$ of degree $d$, there are three of its exponents, say $d_i,d_j$ and $d_k$ for $i<j<k$, such that
$$d_i+d_j+d_k \geq d-1.$$
 \end{cor}  
 \proof
 Consider the graded $S$-module of {\it all Jacobian relations} of $f$ or, equivalently, the module of derivations killing $f$, namely
\begin{equation}
\label{eqD0}
AR(f)= \{\rho=(a^x,a^y, a^z, a^t) \in S^{4} \ : \  a^xf_x+a^yf_y+a^zf_z+a^tf_t=0\}.
\end{equation}
We identify the syzygy $\rho$ above to the derivation
$$\theta=a^x\partial_x+a^y\partial_y+a^z\partial_z+a^t\partial_t$$
such that we have $\theta(f)=0$.
 Note that the vectors $\rho_u \otimes 1$  coming from the minimal set of generators $\rho_u$ for the $S$-module $AR(f)$ span the $K$-vector space $AR(f) \otimes K$. Since this vector space has dimension 3, it follows that there exit $\rho_i,\rho_j$ and $\rho_k$ for $1 \leq i<j<k\leq p$ that yield a basis of $AR(f) \otimes K$. It remains to apply Proposition \ref{propT}.
 \endproof
 
To see how the equality \eqref{t11} extends to the surface case, one may proceed as follows.
If we set $\al_j=c_j-d_{j+3}$ for $j=1,..., p-3$ and $\be_k=b_k-c_{p-3+k}$ for $k=1,\ldots ,r$, it follows from Theorem \ref{thm1} (1) that one has
$$t(X)= \sum_{j=1}^{p-3}\al_j-\sum_{k=1}^r\be_k.$$
\begin{rk}
\label{rk1}
It is not true that for any reduced surface $X$ one has $\al_j \geq 1$ for $j=1,\ldots, p-3$ and $\be_k \geq 1$ for $k=1,\ldots ,r$, see Examples
\ref{ex4} and \ref{ex5}. This is surprising if compared with \eqref{res2B}.
 \end{rk}
 
 \begin{rk}
\label{rk2}
A reduced plane curve is said to be {\it Tjurina maximal} if $2d_1' \geq d$ and 
$$\tau(C)= (d-1)^2-d'_1(d-d'_1-1)-{2d'_1+2-d \choose 2}.$$
It is known that for such a curve one has $\epsilon_j=1$ for all $j=1,\ldots, q'$ and $d_1'=\ldots =d_{p'}$, see \cite[Theorem 3.1]{maxTjurina}.
It is interesting to note that the Cayley and the Kummer surfaces considered in Examples \ref{ex10} and \ref{ex20} enjoy similar properties for their graded Betti numbers, that is the $d_i$'s are all equal, and
$\al_j=\be_k=1$ for $j=1, \ldots, d-3$ and $k=1,\ldots ,r$.
It would be nice to have a theoretical explanation for this fact.
\end{rk}
The nodal surface of degree $d=8$ in Example \ref{ex30} has two values for the degrees
$d_i$'s, namely one has in this case ${\bf d}=(7_6,9_9)$.
A partial explanation of these two values is given by the following result.
Before stating it, we need some notation. 

Inside the $S$-module $AR(f)$, there is the graded $S$-submodule of {\it Koszul type relations} $KR(f)$ generated by the $N$ Koszul type derivations $\theta_{i,j}$ of degree $d-1$, where
$$N=\binom{4}{2}=6$$
and $\{i,j\}$ is an ordered subset of $\{x,y,z,t\}$ having 2 elements and
\begin{equation}
\label{eqK0}
\theta_{i,j}=f_i \partial_j-f_j\partial_i.
\end{equation}
Let $ER(f)=AR(f)/KR(f)$ be the quotient module and let
$$mdr_0(f)=\min \{k \ : \ ER(f)_k \ne 0\}.$$
  \begin{prop} 
\label{propN}
Let $X:f=0$ be a nodal surface in $\PP^3$ of degree $d \geq 5$.
Then
$$mdr_0(f) \geq 2d-  \floor[\Big]{ \frac{d}{2} } -3 > d-1.$$
In particular, the graded Betti numbers $\bf d$ of the surface $X$ satisfy
$$d_i=d-1 \text{ for } i=1, \ldots, N$$
and
$$d_i \geq 2d-   \floor[\Big]{ \frac{d}{2} } -3$$
for $i>N$.
\end{prop}
\proof
The inequality
$$mdr(f) \geq 2d-  \floor[\Big]{ \frac{d}{2} } -3 $$
follows from \cite[Equation (4.5) and Corollary 4.3]{Ferrara}.
The inequality 
$$2d-  \floor[\Big]{ \frac{d}{2} } -3 \geq d-1$$
clearly holds for any $d \geq 5$.
Since the $N$ Koszul type derivations $\theta_{ij}$ are linearly independent in $AR(f)_{d-1}$, we need to have exactly $N$ elements of degree $d-1$ in any minimal set of generators for $AR(f)$. The other generators we add yield nonzero elements in $ER(f)$, and hence their degree is at least $mdr(f)$.
\endproof

\begin{rk}
\label{rk3}
(i) Note that in Example \ref{ex30} we have a nodal surface of degree $d=8$ and the corresponding module $AR(f)$ has $N=6$ generators of degree $d-1=7$ and the remaining generators have degree 
$d_i=9$ for $i>6$. Since in this case
$$2d-  \floor[\Big]{ \frac{d}{2} } -3=9,$$
we see that our Proposition \ref{propN} is sharp.

(ii) A similar result to Proposition \ref{propN} may be stated in the more general situation when the surface $X$ has only isolated weighted homogeneous singularities, by using \cite[Theorem 9]{DSa}, which extends
the results in \cite{Ferrara} to this setting.
\end{rk}
For a curve $C:f'=0$ of degree $d$ as above which is not free it is known that 
 $$d_1' \leq d_2' \leq d_3' \leq d-1,$$
 see for instance \cite[Theorem 2.4]{3syz}.   This is the motivation for Conjecture \ref{C1} stated in the Introduction.
  
 \begin{rk}
\label{rkC1}
Conjecture \ref{C1} holds for the following three types of surfaces.

\medskip

\noindent (i) Let $C:f'(x,y,z)=0$ be a reduced curve of degree $d'>1$  which is not free such that $d_1'>0$ and consider the surface
$$X:f(x,y,z,t)=tf'(x,y,z)=0$$
of degree $d=d'+1>2$. If  $1 \leq d'_1\leq ...\leq d'_{p'}$ are the exponents of $C$, then it follows from \cite[Theorem 1.6]{DPZPPA} that the exponents of $X$ are
$$d_1=1 \text{ and } d_k=d'_{k-1}$$
for $2 \leq k \leq p=p'+1$.  Since
$$d_1' \leq d_2' \leq d_3' \leq d'-1,$$
it follows that
$$d_1 \leq d_2 \leq d_3 \leq d_4 \leq d'-1=d-2.$$
%\medskip

\noindent (ii) Let $C:f'(x,y,z)=0$ be a reduced curve of degree $d>1$  such that $d_1'>0$ and consider the surface
$$X:f(x,y,z,t)=f'(x,y,z)+t^d=0$$
of degree $d$. If  $1 \leq d'_1\leq ...\leq d'_{p'}$ are the exponents of $C$, it is easy to see that the exponents of $X$ are
$$d_k=d_k' \text{ for } k=1, \ldots ,p'$$
and $d_k=d-1$ for $k=p'+1, \ldots, p'+3$, the new generators to be added being $\theta_{x,t}$, $\theta_{y,t}$ and $\theta_{z,t}$ with the notation from \eqref{eqK0}.
It follows that Conjecture \ref{C1} holds in this case as well.

\medskip

\noindent (iii) Conjecture \ref{C1} holds for any nodal surface $X$ as in Proposition \ref{propN}.

 \end{rk}
 
 We can prove the following weaker form of Conjecture \ref{C1}.

  \begin{prop} 
\label{propC1}
 For any reduced surface $X:f=0$ in $\PP^3$ of degree $d$  one has
 $$d_1 \leq d_2 \leq d_3  \leq d-1.$$
 \end{prop}  
 \proof
 Note that the Koszul derivations $\theta_{x,t}$, $\theta_{y,t}$ and $\theta_{z,t}$ with the notation from \eqref{eqK0} are linearly independent over the field of fractions $K$. Hence they form a basis of the vector space $AR(f) \otimes K$, if we identify $\theta \in AR(f)$ and $\theta \otimes 1 \in AR(f) \otimes K$. If $d_3 >d-1$, then it follows that 
the Koszul derivations $\theta_{x,t}$, $\theta_{y,t}$ and $\theta_{z,t}$ are linear combinations  with coefficients in $S$ of the generators $\rho_1, \rho_2 $ of $ AR(f)$,
with $\deg \rho _i=d_i$ for $i=1,2$. It follows that the 3-dimensional vector space 
$AR(f) \otimes K$ is generated by $\rho_1$ and $\rho_2$, a contradiction. 
 \endproof
 
  \begin{rk}
\label{rkC10}
Note that Theorem \ref{thmC} and Theorem \ref{thmD} hold with exactly the same statement when we replace the surface $X$ in $\PP^3$ by a hypersurface $X$ in $\PP^n$ for any $n \geq 3$, see \cite[Theorem 5.3]{duPCTC01} and respectively \cite[Lemma 5.2]{duPCTC01}.
Similarly, Proposition \ref{propC1} holds for any hypersurface $X$ in $\PP^n$ for any $n\geq 3$, the conclusion being replaced by
$$d_1 \leq d_2 \leq \ldots \leq d_n \leq d-1.$$
The proof of this claim is exactly as the proof of Proposition \ref{propC1}.
Indeed, if $(x_0, \ldots,x_n)$ are the coordinates on $\PP^n$, then the
Koszul derivatives 
$$\theta_{i,n}=f_i \partial_n-f_n\partial_i \text{ for } i=0, \ldots,n-1$$
are linearly independent over the corresponding fraction field $K$.
  \end{rk}

 \section{ Some examples}

\begin{ex}
\label{ex2}
We start with one example of a surface with isolated singularities, where the invariants  $\al_i$ and $\beta_j$ behave in a similar way to the curve case.
Consider the cubic surface
$$X: f=txz + y^2z+x^3-z^3=0.$$
A direct computation  shows that
the corresponding minimal resolution for $M(f)$ is given by
$$0 \to S(-6) \to  S(-5)^5 \to S(-3) \oplus  S(-4)^6 \to  S(-2)^4 \to S.$$
It follows that $p=7$, $q=5$ and $r=1$, and the graded Betti numbers of this surface are
$$ {\bf d}=(1,2_6),  \ {\bf c}=(3_5) \text{ and } {\bf b}=(4).$$
The surface has  $\tau(X)=5$.
All the claims in Theorem \ref{thm1} hold and  one has
$$\al_1=c_1-d_4=3-2=1, \  \al_2=c_2-d_5=3-2=1, \  \al_3=c_3-d_6=3-2=1, $$
$$ \al_4=c_4-d_7=3-2=1, \be_1=b_1-c_5=4-3=1.$$
\end{ex}

\begin{ex}
\label{ex3}
Consider the sextic surface
$$X: f=x^5z+y^6+x^4yt+xy^5=0.$$
A direct computation  shows that
the corresponding minimal resolution for $M(f)$ is given by
$$0  \to S(-9) \to S(-6) \oplus S(-7) \oplus  S(-8)^2 \to  S(-5)^4 \to S.$$
It follows that $p=4$, $q=1$ and $r=0$, and the graded Betti numbers of this surface are
$$ {\bf d}=(1,2,3_2) \text{ and } {\bf c}=(4).$$
The surface $X$ is nearly-free and has a 1-dimensional singular locus, with the Hilbert polynomial
$$P(M(f))(u)=16u-27.$$
All the claims in Theorem \ref{thm1} hold and one has
$\al_1=c_1-d_4=4-3=1. $
\end{ex}

\begin{ex}
\label{ex4}
Consider the surface of degree $d=9$ given by
$$X: f=(x^3-yzt)^3+(t^3-xyz)^3=0.$$
Then $X$  is a union of 3 cubic surfaces in $\PP^3$ and
a direct computation  shows that
the corresponding minimal resolution for $M(f)$ is given by
$$0 \to S(-18)\oplus S(-19)  \to S(-13)\oplus S(-16)\oplus S(-17)^3 \oplus S(-18)^2 \to $$
$$\to S(-9)  \oplus  S(-12)^2 \oplus S(-15)^2 \oplus S(-16)^3 \to  S(-8)^4 \to S.$$
It follows that $p=8$, $q=7$ and $r=2$, and the graded Betti numbers of this surface are
$$ {\bf d}=(1,4_2,7_2,8_3), \  {\bf c}=(5,8,9_3,10_2)  \text{ and } {\bf b}=(10,11).$$
All the claims in Theorem \ref{thm1} hold, in particular one has
$$P(M(f))(u)=38u-119.$$
With the notation from Remark \ref{rk1} one has
$$\be_1=b_1-c_6=10-10=0. $$
This example is revisited below in Example \ref{ex6.1}.
\end{ex}

\begin{ex}
\label{ex5}
Finally one example where the invariants $\al_i$ and $\beta_j$ have negative values.
Consider the surface of degree $d=16$ given by
$$X: f=(x^4-yzt^2)^4+(t^4-xyz^2)^4=0.$$
Then $X$  is a union of 4 quartic surfaces in $\PP^3$ and
a direct computation  shows that
the corresponding minimal resolution for $M(f)$ is given by
$$0 \to S(-29)\oplus S(-30)\oplus S(-32)  \to S(-22)^2\oplus S(-23)\oplus S(-28)^4 \oplus S(-29)^3 \oplus S(-31)^2 \to $$
$$\to S(-20)^2  \oplus  S(-21)^3 \oplus S(-27)^4 \oplus S(-28)^2 \oplus S(-30) \to  S(-15)^4 \to S.$$
It follows that $p=12$, $q=12$ and $r=3$, and the graded Betti numbers of this surface are
$$ {\bf d}=(5_2,6_3,12_4,13_2,15), \  {\bf c}=(7_2,8,13_4,14_3,16_2)  \text{ and } {\bf b}=(14,15,17).$$
All the claims in Theorem \ref{thm1} hold, in particular one has
$$P(M(f))(u)=147u-1382.$$
With the notation from Remark \ref{rk1} one has
$$\al_9=c_9-d_{12}=14-15=-1,$$
$$\be_1=b_1-c_{10}=14-14=0,$$
and
$$\be_2=b_2-c_{11}=15-16=-1. $$
\end{ex}

 \section{ On the Jacobian syzygies of surfaces coming from pencils of surfaces in $\PP^3$}
 
 Let $g,h \in S$ be two homogeneous polynomials of degree $k\geq 2$ and consider the pencil of surfaces in $\PP^3$ given by
 $$\PPP : \al g +\be h=0,$$
where $\al, \be \in \C$.  Consider a surface
 $$X:f=q_1 \ldots q_m=0,$$
 for some integer $m \geq 2$, where $q_j=\al _jg +\be_j h$ are reduced members of the pencil $\PPP$
 considered above. Then $X$ is a reduced surface of degree $d=km$, which has  {\it four natural Jacobian syzygies of degree $2k-2$ }that we describe now. Recall that a syzygy
 \begin{equation}
\label{eqS1}
\rho=(a^x,a^y,a^z,a^t) \text{ corresponding to } a^xf_x+a^yf_y+a^zf_z+a^tf_t=0,
 \end{equation} 
with $a^x,a^y,a^z,a^t \in S$ may be identified to a differential 3-form
 \begin{equation}
\label{eqS2}
\omega(\rho)= \ a^x \dd y\wedge \dd z  \wedge \dd t-a^y  \dd x\wedge \dd z  \wedge \dd t+a^z \dd x\wedge \dd y  \wedge \dd t-a^t \dd x\wedge \dd y  \wedge \dd z,
 \end{equation}  
satisfying
 \begin{equation}
\label{eqS3}
\dd f \wedge \omega(\rho)= 0
 \end{equation} 
 with $\dd f=f_x \dd x+f_y \dd y +f_z \dd z +f_t \dd t$ the differential of $f$.
 The differential 2-form
 \begin{equation}
\label{eqS4}
\omega(\PPP)=\dd g \wedge \dd h
 \end{equation} 
 clearly satisfies
  \begin{equation}
\label{eqS5}
\dd f \wedge \omega(\PPP)=0,
 \end{equation} 
 a key fact that was used with a similar purpose in the plane curve case in \cite{Mich}. It follows that the following four 3-forms 
  \begin{equation}
\label{eqS6}
\omega^x=\dd x \wedge \omega(\PPP), \  \omega^y=\dd y \wedge \omega(\PPP), \ \omega^z=\dd z \wedge \omega(\PPP) \text{ and } \omega^t=\dd t \wedge \omega(\PPP)
 \end{equation}  
 yield four Jacobian syzygies for the surface $X$. We have the following result, showing that these syzygies are in general linearly independent.
  \begin{thm} 
\label{thmPPP} 
 Assume that the base locus
 $$B=\{g=h=0\}$$
 of the pencil $\PPP$ with $k=\deg g=\deg h \geq 2$ is a reduced curve. If a linear dependence relation
 $$c^x\omega^x+c^y\omega^y+c^z\omega^z+c^t\omega^t=0$$
 with $c^x,c^y,c^z,c^t \in \C$ holds, then the curve $B$ is contained in the plane
 $$H: \ell=c^x x+c^yy+c^zz+c^tt=0.$$
 In particular, the differential forms $\omega^x, \omega^y, \omega^z$ and $\omega^t$ span a vector subspace of dimension  $3$ or $4$.
 \end{thm} 
 \proof
First note that the linear dependence relation is equivalent to the equality
  \begin{equation}
\label{eqS7}
\dd \ell \wedge \omega(\PPP)=0,
 \end{equation} 
with $ \dd \ell$ the differential of $\ell$. Let $B_1$ be an irreducible component of the base locus $B$ and let $s$ be a smooth point on $B_1$. It follows that the differentials $\dd_s g$ and $\dd_sh$ are linearly independent linear forms. The condition \eqref{eqS7} is equivalent to saying that $\ell$, identified to $\dd_s \ell$, is a linear combination of $\dd_s g$ and $\dd_sh$. Since $\dd_s g(s)=\dd_sh(s)=0$, it follows that $\ell(s)=0$. Since $B_1$ is a reduced curve, the smooth points are dense in $B_1$, hence $B_1$ is contained in the hyperplane $H$. This proof applies to any irreducible component of $B$, hence $B \subset H$. 

If $k>1$ then $B$ is a complete intersection of degree $k^2 >1$, hence $B$ cannot be a line. But a reduced curve contained in a line is a line. It follows that there cannot be two independent linear relations involving the 4 differential forms, since otherwise
$B \subset H_1 \cap H_2$, where $H_1$ and $H_2$ are two distinct planes in $\PP^3$. 
\endproof 

 The following example shows that the four Jacobian syzygies constructed in Theorem \ref{thmPPP} may not be {\it primitive}, that is they can be the multiples of  lower degree Jacobian syzygies.
 \begin{ex}
\label{ex6.1}
Consider the surface of degree $d=3m$ given by
$$X: f=(x^3-yzt)^m+(t^3-xyz)^m=0,$$
for any $m \geq 2$, that extends Example \ref{ex4} above.
Then $X$  is a union of $m\geq 2$ cubic surfaces in $\PP^3$ that are reduced members of a pencil $\PPP$ as above with 
$$g=x^3-yzt \text{ and } h=t^3-xyz.$$
A direct computation shows that
$$
\omega(\PPP)=\dd g \wedge \dd h=-z(3x^3+yzt)\dd x \wedge \dd y-y(3x^3+yzt)\dd x \wedge \dd z-$$
$$-(y^2z^2-9x^2t^2)\dd x \wedge \dd t
-
z(3t^3+xyz)\dd y \wedge \dd t-y(3t^3+xyz)\dd t \wedge \dd t.$$
It follows that
$$\omega^x=\dd x \wedge \omega(\PPP)=-(3t^3+xyz)(z\dd x \wedge \dd y \wedge \dd t+y\dd x \wedge \dd z \wedge \dd t).$$
Using the formulas \eqref{eqS1} and \eqref{eqS2}, we see that the Jacobian syzygy $\rho^x$ corresponding to the differential form $\omega^x$ is
$$\rho^x=(3t^3+xyz)\rho_1$$
where 
$$\rho_1=(0,y,-z,0)$$
is the lowest degree syzygy of $f$. A similar computation shows that
the Jacobian syzygy $\rho^t$ corresponding to the differential form $\omega^t$ is
$$\rho^x=(3x^3+yzt)\rho_1.$$ 
Moreover, the Jacobian syzygies $\rho^y$ and $\rho^z$ corresponding to the differential form $\omega^y$ and respectively
$\omega^z$ are given by the following formulas
$$\rho^y=(-y(3t^3+xyz),0,y^2z^2-9x^2t^2,-y(3x^3+yzt))$$
and
$$\rho^z=(z(3t^3+xyz),9x^2t^2-y^2z^2,0,z(3x^3+yzt)).$$
We see that the syzygies $\omega^y$ and 
$\omega^z$ are primitive, since in each case the four coordinate polynomials have no common factor. Moreover, the syzygies of the form
$$\rho=c^y \rho^y +c^z \rho^z + b \rho_1$$
with $c^y,c^z \in \C$ not both zero, and $b \in S_3$ are all primitive.
Indeed, looking at the first and the fourth coordinate polynomials of $\rho$, we see that the only possible common factor is a linear form in $y,z$, but such a factor cannot divide the second and the third coordinates of $\rho$, due to the monomial $x^2t^2$ present there.
It follows that one can take $\rho_1$, $\rho_2=\rho^y$ and $\rho_3=\rho^z$ as the first 3 generating syzygies for $AR(f)$ and hence $d_1=1$ and $d_2=d_3=4$ for any $m \geq 2$.
This shows that
$$t(X)=1+4+4+1-3m=10-3m$$
can be an arbitrarily large negative integer.
In this example the vector space spanned by the forms $\omega^x, \omega^y, \omega^z$ and $\omega^t$ has dimension 4, as predicted by Theorem \ref{thmPPP}. On the other hand, the obvious relation
$$(9x^2t^2-y^2z^2)\rho_1+z\rho_2+y\rho_3=0$$
explain why Proposition \ref{propT} cannot be applied to these 3 syzygies and get $t(X)\geq 0$.

\end{ex}

%\section*{Conflict of Interests}
%We declare that there is no conflict of interest regarding the publication of this paper.

%\section*{Data Availability Statement}
%We do not analyze or generate any datasets, because this work proceeds within a theoretical and mathematical approach.


\begin{thebibliography}{00}

%\bibitem{Abe}  T. Abe, Plus-one generated and next to free arrangements of hyperplanes. \textit{Int. Math. Res. Not.} \textbf{2021(12)}: 9233 -- 9261 (2021).



\bibitem{CoCoA}
    {J. Abbott, A. M.  Bigatti  and L. Robbiano},
    {CoCoA  4.7.4 }: a system for doing {C}omputations in {C}ommutative {Algebra}. Available at {http://cocoa.dima.unige.it}
    
 \bibitem{ADP}  T. Abe, A. Dimca, P. Pokora, A new hierarchy for complex plane curves,
Canadian Mathematical Bulletin. Published online 2025:1-24. doi:10.4153/S0008439525101422   
    
    
\bibitem{AGV} V. I. Arnold, S. M.  Gusein-Zade, A. N. Varchenko,  Singularities of Differentiable Maps. vols 1/2, Monographs in Math., \textbf{ 82/83}, Birkh\"auser, Basel (1985/1988)    

\bibitem{Hess}  L. Bus\'e, A. Dimca, H. Schenck and G. Sticlaru, The Hessian polynomial and the Jacobian ideal of a reduced hypersurface in $\PP^n$, Advances in Mathematics,
Volume 392 (2021), 108035.
    
\bibitem{Chm}  S. V. Chmutov, Examples of projective surfaces with many singularities, J. Algebr. Geom. 1, 191--196 (1992).

%\bibitem{CD} A.D.R. Choudary, A. Dimca,  Koszul complexes and hypersurface singularities, Proc. Amer. Math. Soc. 121 (1994), 1009--1016.

\bibitem{Singular} { W. Decker, G.-M. Greuel, G. Pfister \and H. Sch{\"o}nemann.} \newblock {\sc Singular} {4-0-2} --- {A} computer algebra system for  polynomial computations. Available at {http://www.singular.uni-kl.de}.

\bibitem{DmaxS} A. Dimca, Freeness versus maximal degree of the singular subscheme for surfaces in $P^3$, 
Geom. Dedicata 183(2016), 101--112.

\bibitem{Dmax}  A. Dimca, Freeness versus maximal global Tjurina number for plane curves. \textit{Math. Proc. Cambridge Phil. Soc.}  \textit{163}:  161 -- 172 (2017).

\bibitem{Mich}  A. Dimca, Curve arrangements, pencils, and Jacobian syzygies,  Michigan Math. J. 66 (2017), 347--365.

\bibitem{Ferrara}  A. Dimca, 
On the syzygies and Hodge theory of nodal hypersurfaces, Ann. Univ. Ferrara Sez. VII Sci. Mat. 63 (2017), 87--101.

%\bibitem{CMreg} A. Dimca, On the Castelnuovo-Mumford regularity of curve arrangements.  \textit{Rev. Roumaine Math. Pures Appl}. 70 (2025),  73--84.

\bibitem{DPZPPA}  A. Dimca, P. Pokora,  On the Jacobian algebras of Ziegler pairs of plane arrangements, arXiv:2604.25637.

\bibitem{DSa} A. Dimca, M. Saito,
Generalization of theorems of Griffiths
and Steenbrink to hypersurfaces
with ordinary double points, Bull. Math. Soc. Sci. Math. Roumanie, 60(108) (2017), 351--371.

\bibitem{DStChm} A. Dimca, G. Sticlaru, On the syzygies and Alexander polynomials of nodal hypersurfaces,  Math. Nachrichten 285 (2012), 2120--2128.

\bibitem{DStExpo} A. Dimca, G. Sticlaru, On the exponents of free and nearly free projective plane curves, Rev. Mat. Complut. 30(2017), 259--268.

%\bibitem{DStRIMS} A. Dimca, G. Sticlaru, Free and nearly free curves vs. rational cuspidal plane curves. \textit{Publ. RIMS Kyoto Univ.} \textbf{54}: 163 -- 179 (2018).

\bibitem{3syz} A. Dimca and G. Sticlaru, Plane curves with three syzygies, minimal Tjurina curves curves, and nearly cuspidal curves. \textit{Geom. Dedicata} \textbf{207}: 29 -- 49 (2020).

\bibitem{maxTjurina} A. Dimca, G. Sticlaru, Jacobian syzygies, Fitting ideals, and plane curves with maximal global Tjurina numbers, Collect. Math. 
73 (2022), 391--409.

%\bibitem{Brian} A. Dimca, G. Sticlaru, Plus-one generated curves, Briançon-type polynomials and eigenscheme ideals,   Results Math.(2025) https://doi.org/10.1007/s00025-025-02371-z
 
%\bibitem{SameDeg} A. Dimca, G. Sticlaru, Curves with Jacobian syzygies of the same degree, Rendiconti del Circolo Matematico di Palermo Series 2, (2025) https://doi.org/10.1007/s12215-025-01203-x

\bibitem{HD}
A.~Dimca and G.~Sticlaru,
Bourbaki modules and the module of Jacobian derivations of projective hypersurfaces, Collectanea Mathematica
https://doi.org/10.1007/s13348-026-00509-y

\bibitem{Type3}
A.~Dimca and G.~Sticlaru,
On type three complex plane curves,
arXiv:2601.01824.

 

\bibitem{duPCTC} A.A. du Plessis,  C.T.C. Wall, Application of the theory of the discriminant
to highly singular plane curves, Math. Proc. Camb. Phil. Soc. 126
(1999) 256--266.

%\bibitem{duPCTC00} A.A. du Plessis,  C.T.C. Wall,  Singular hypersurfaces, versality and Gorenstein algebras, Jour. Alg. Geom. 9 (2000) 309--322.

\bibitem{duPCTC01} A.A. du Plessis,  C.T.C. Wall, Discriminants, vector fields and singular hypersurfaces. New developments in singularity theory (Cambridge, 2000), 351--377, NATO Sci. Ser. II Math. Phys. Chem., 21, Kluwer Acad. Publ., Dordrecht, 2001.

 \bibitem{duPCTC08} A.A. du Plessis,  C.T.C. Wall, Hypersurfaces with isolated singularities with symmetry. In: Saia, J.M., Seade, J. (eds.) Real and Complex Singularities. (Proceedings of the IX International Workshop). Contemp. Math. Amer. Math. Soc. 459, pp 147–164 (2008).

\bibitem{Eis} { D. Eisenbud}, \emph{The Geometry of Syzygies: A Second Course in Algebraic Geometry and Commutative Algebra}, Graduate Texts in Mathematics, Vol. 229, Springer 2005. 

\bibitem{E} Ph. Ellia,  Quasi complete intersections and global Tjurina number of plane curves, J. Pure Appl. Algebra 224 (2020),  423--431.

\bibitem{HS} S. H. Hassanzadeh, A. Simis, Plane Cremona maps: Saturation and regularity of the base ideal. \textit{J. Algebra} \textbf{371}: 620 -- 652 (2012).

\end{thebibliography}
\end{document}